\documentclass[]{theclass} 
\usepackage{graphicx, xfrac, lineno, float, subcaption, tasks, comment, xcolor, booktabs, multirow}
\usepackage[T1]{fontenc}
\usepackage[normalem]{ulem}
\usepackage{datetime}
\usepackage{colortbl}
\usepackage{enumitem}

\begin{document}
\begin{frontmatter}

\titledata{Disjoint odd circuits in a bridgeless cubic graph can be quelled by a single perfect matching}{To our loved ones \\ The authors were partially supported by VEGA 1/0743/21, VEGA 1/0727/22, and APVV-19-0308. The results described in this paper were obtained during the workshop Exceptional Structures in Discrete Mathematics 4 in Modra, Slovakia.}  

\authordata{Franti\v{s}ek Kardo\v{s}}
{LaBRI, CNRS, University of Bordeaux, Talence, F-33405, France; \\ Department of Computer Science, Faculty of Mathematics, Physics and Informatics\\ Comenius University, Mlynsk\'{a} Dolina, 842 48 Bratislava, Slovakia}
{frantisek.kardos@u-bordeaux.fr}{}
\authordata{Edita M\'{a}\v{c}ajov\'{a}}
{Department of Computer Science, Faculty of Mathematics, Physics and Informatics\\ Comenius University, Mlynsk\'{a} Dolina, 842 48 Bratislava, Slovakia}
{macajova@dcs.fmph.uniba.sk}{}
\authordata{Jean Paul Zerafa}
{St. Edward's College, Triq San Dwardu\\ Birgu (Citt\`{a} Vittoriosa), BRG 9039, Cottonera, Malta; \\ Department of Computer Science, Faculty of Mathematics, Physics and Informatics\\ Comenius University, Mlynsk\'{a} Dolina, 842 48 Bratislava, Slovakia}
{zerafa.jp@gmail.com}
{}


\begin{abstract}
Let $G$ be a bridgeless cubic graph. The Berge--Fulkerson Conjecture (1970s) states that $G$ admits a list of six perfect matchings such that each edge of $G$ belongs to exactly two of these perfect matchings. If answered in the affirmative, two other recent conjectures would also be true: the Fan--Raspaud Conjecture (1994), which states that $G$ admits three perfect matchings such that every edge of $G$ belongs to at most two of them; and a conjecture by Mazzuoccolo (2013), which states that $G$ admits two perfect matchings whose deletion yields a bipartite subgraph of $G$. It can be shown that given an arbitrary perfect matching of $G$, it is not always possible to extend it to a list of three or six perfect matchings satisfying the statements of the Fan--Raspaud and the Berge--Fulkerson conjectures, respectively. In this paper, we show that given any $1^+$-factor $F$ (a spanning subgraph of $G$ such that its vertices have degree at least 1) and an arbitrary edge $e$ of $G$, there always exists a perfect matching $M$ of $G$ containing $e$ such that $G\setminus (F\cup M)$ is bipartite. Our result implies Mazzuoccolo's conjecture, but not only. It also implies that given any collection of disjoint odd circuits in $G$, there exists a perfect matching of $G$ containing at least one edge of each circuit in this collection. 
\end{abstract}

\end{frontmatter}

\section{Introduction}

The behaviour of perfect matchings in cubic graphs is amongst the most well-studied themes in graph theory. One of the first classical results was made by Petersen in 1891 \cite{petersen}, who showed that every bridgeless cubic graph admits a perfect matching. These graphs not only do admit a perfect matching, but in 2011, one of the most prominent conjectures about perfect matchings in bridgeless cubic graphs was completely solved by Esperet \emph{et al.} \cite{EsperetLovaszPlummer}. The conjecture, proposed by Lov\'{a}sz and Plummer in the 1970s, stated that the number of perfect matchings in a bridgeless cubic graph grows exponentially with its order (see \cite{LovaszPlummer}). Another conjecture which has baffled mathematicians for more than 50 years is the following.

\begin{conjecture}[Fulkerson, 1971 \cite{BergeFulkerson}]\label{conjecture fulkerson}
Every bridgeless cubic graph $G$ admits six perfect matchings such that each edge in $G$ is contained in exactly two of these six perfect matchings.
\end{conjecture}

Such a list of six perfect matchings shall be referred to as a \emph{Fulkerson cover}. Although initially stated by Fulkerson, Conjecture \ref{conjecture fulkerson} is also attributed to Berge and has been widely referred to as the Berge--Fulkerson Conjecture. In hindsight, such a name was more than appropriate as Fulkerson's conjecture was actually shown to be equivalent to the following seemingly weaker conjecture made by Berge (see \cite{MazzuoccoloEquivalence} for more details).

\begin{conjecture}[Berge, unpublished]\label{conjecture berge}
Every bridgeless cubic graph $G$ admits five perfect matchings such that every edge in $G$ is contained in at least one of these five perfect matchings.
\end{conjecture}

In the sequel, in order to avoid confusion, the Berge--Fulkerson Conjecture shall refer exclusively to the statement by Fulkerson, which, despite being quite simple and uncomplicated, remains widely open. As the years went by, in an attempt to advance in this direction, weaker assertions started to be considered. Such an example is a conjecture by Fan and Raspaud.

\begin{conjecture}[\mbox{Fan--Raspaud}, 1994 \cite{FanRaspaud}]\label{conjecture FR}
Every bridgeless cubic graph admits three perfect matchings $M_{1}, M_{2},M_{3}$ such that $M_{1}\cap M_{2}\cap M_{3}=\emptyset$.
\end{conjecture} 

Three perfect matchings satisfying this property shall be referred to as an \emph{FR-triple}. As one can readily deduce, any three perfect matchings from a Fulkerson cover form an FR-triple. Although, as stated by the authors in \cite{FanRaspaud}, this conjecture is weaker than the Berge--Fulkerson Conjecture; once again, it remains unresolved. The best result in this direction is a result by the second author and \v{S}koviera \cite{MacajovaSkovieraOddness2} who show that the Fan--Raspaud Conjecture is true for bridgeless cubic graphs having oddness 2. The \emph{oddness} of a bridgeless cubic graph $G$ is the least (even) number of odd circuits in a $2$-factor amongst all possible $2$-factors of $G$. Seemingly stronger formulations of the Fan--Raspaud Conjecture are discussed in \cite{s4 gm jp, Vahan, zerafa thesis}, and shown to be equivalent to the original formulation.

A \emph{Fano colouring} of a cubic graph $G$ is a colouring of the edges of $G$ with the points of the Fano plane such that the colours given to the three edges incident to a vertex of $G$ form a line of the Fano plane, where the latter consists of seven points and seven lines (see Figure \ref{fig:FanoS4}(a)). A cubic graph admits a Fano colouring if and only if it is bridgeless (see \cite{HolroydSkoviera}). In \cite{MacajovaSkovieraOddCuts}, the second author and \v{S}koviera  deal with the question of how many points and lines of the Fano plane are sufficient to colour all bridgeless cubic graphs. It is shown that for a Fano colouring of a snark (a non-3-edge-colourable bridgeless cubic graph), all seven points and at least four lines are needed. In fact, it is proven that six, and conjectured that four, lines of the Fano plane colour every bridgeless cubic graph. For $i\in\{4,5\}$, the statement that for every bridgeless cubic graph there exists a Fano colouring with at most $i$ lines is denoted as the $\mathcal{F}_i$-Conjecture. It is shown that the $\mathcal{F}_4$-Conjecture is equivalent to the Fan--Raspaud Conjecture, and that the $\mathcal{F}_5$-Conjecture is equivalent to saying that every bridgeless cubic graph $G$ admits two perfect matchings whose intersection does not contain any odd edge-cut of $G$. In 2013, Mazzuoccolo proposed the following conjecture, which is equivalent to the $\mathcal{F}_5$-Conjecture restricted to only those odd edge-cuts that separate odd circuits.

\begin{figure}
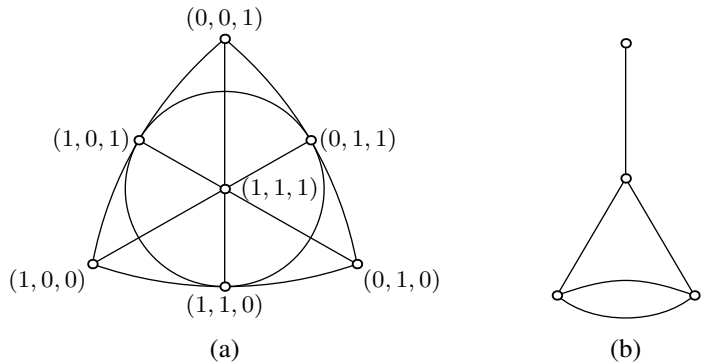

    \centering
    \setlength{\tabcolsep}{20pt}
    \begin{tabular}{cc}
    \includegraphics{FanoS4.0} &
    \includegraphics{FanoS4.1} \\
        (a) &  (b)
    \end{tabular}
    \caption{(a) The Fano plane (the configuration where points are non-zero elements of $\mathbb{Z}_2^3$ and lines are triples that sum up to zero); (b) The graph $S_4$.}
    \label{fig:FanoS4}
\end{figure}

\begin{conjecture}[Mazzuoccolo, 2013 \cite{MazzuoccoloS4}]\label{conjecture s4 matchings}
For any bridgeless cubic graph $G$, there exist two perfect matchings such that the deletion of their union leaves a bipartite subgraph of $G$. 
\end{conjecture}

Let $H$ and $G$ be graphs. An \emph{$H$-colouring} of $G$ is a map $f:E(G)\rightarrow E(H)$ such that for any vertex $u\in V(G)$ there exists a vertex $v\in V(H)$ with $f\left (\partial_G(u)\right )=\partial_H(v)$, where $\partial_G(u)$ and $\partial_H(v)$ respectively denote the sets of edges in $G$ and $H$ incident to the vertices $u$ and $v$. Conjecture \ref{conjecture s4 matchings} was later referred to as the $S_{4}$-Conjecture by Mazzuoccolo and the last author, since it is equivalent to colouring the edges of $G$ with the edges of the graph $S_{4}$ shown in Figure \ref{fig:FanoS4}(b) (see \cite{gmgtjp, s4 gm jp, zerafa thesis} for more details). Consequently, as in the last two cited documents, given a bridgeless cubic graph $G$, a pair of perfect matchings whose deletion yields a bipartite subgraph of $G$ shall be referred to as an \emph{$S_{4}$-pair}. 
 
In 2013, Mkrtchyan \cite{VahanPetersen} showed that if the Petersen Colouring Conjecture by Jaeger \cite{Jaeger} is true, the Petersen graph is the unique connected bridgeless cubic graph which can colour all bridgeless cubic graphs (colouring is done in the same manner as $H$-colourings defined above). In \cite{gmgtjp} this was extended and it was shown that if one was to remove all degree conditions on the graph $H$ by which we are colouring, any bridgeless cubic graph can be coloured by the Petersen graph, and by a unique other graph: $S_{4}$ (see \cite{gmgtjp} for more details about the uniqueness of $S_4$). Related results can be found in \cite{vh1, vh2}.

Clearly, by the result in \cite{MacajovaSkovieraOddness2}, any bridgeless cubic graph having oddness 2 admits an $S_{4}$-pair. The authors of \cite{s4 gm jp, zerafa thesis} extend this result to bridgeless cubic graphs having oddness 4. The authors also show that if the $S_{4}$-Conjecture is true, then there exist cubic graphs admitting a bridge for which an $S_{4}$-pair exists -- this is not possible when considering the Berge--Fulkerson Conjecture and the weaker conjectures mentioned above, because if a cubic graph admitting a perfect matching contains a bridge, then, every perfect matching of this graph must intersect the bridge (itself an odd edge-cut).

\begin{theorem}[\cite{s4 gm jp, zerafa thesis}]\label{theorem S4 for graphs with bridges}
Let $G$ be a connected cubic graph having $k$ bridges, all of which lie on a single path, for some positive integer $k$. If the $S_{4}$-Conjecture is true, then $G$ admits an $S_{4}$-pair.
\end{theorem}

\subsection{Important definitions and notation}

Graphs considered in this paper may contain parallel edges, but they cannot contain loops, unless otherwise stated. We speak about a \emph{simple} graph if parallel edges are not allowed.

Let $G$ be a graph and $(V_1,V_2)$ be a partition of its vertex set, that is, $V_1\cup V_2=V(G)$ and $V_1\cap V_2=\emptyset$. Then, by $E(V_1,V_2)$ we denote the set of edges having one endvertex in $V_1$ and one in $V_2$; we call such a set an \emph{edge-cut}. An edge which itself is an edge-cut of size one is a \emph{bridge}.

An edge-cut $X=E(V_1,V_2)$ is called \emph{cyclic} if both graphs $G[V_1]$ and $G[V_2]$, obtained from $G$ after deleting $X$, contain a \emph{circuit} (a 2-regular connected subgraph). 
The \emph{cyclic edge-connectivity} of a graph $G$ is defined as the smallest size of a cyclic edge-cut in $G$ if $G$ admits one; it is defined as $|E(G)|-|V(G)|+1$ otherwise. For cubic graphs, the latter only concerns three graphs: $K_4$, $K_{3,3}$, and the graph consisting of two vertices joined by three parallel edges, whose cyclic edge-connectivity is thus 3, 4, and 2, respectively.

Let $G$ be a bridgeless cubic graph. A \emph{$1^+$-factor} of $G$ is the edge set of a spanning subgraph of $G$ such that its vertices have degree 1, 2 or 3. In particular, a \emph{perfect matching} and a \emph{$2$-factor} of $G$ are $1^+$-factors whose vertices have exactly degree 1 and 2, respectively. A \emph{parity subgraph} $J$ of a graph $G$ is a subgraph $H$ such that the parity of the degree of every vertex is the same in $H$ as is in $G$. For cubic graphs, a parity subgraph is a $1^+$-factor with no vertices of degree 2.
We remark that in \cite{kaiserraspaud}, parity subgraphs (also known as \emph{joins} in the literature) have been already studied with respect to problems discussed above.  

\begin{definition}
A $1^+$-factor $F$ and a perfect matching $M$ of $G$ form a \emph{quelling pair} if $G\setminus (F\cup M)$ is bipartite.
\end{definition}

\section{Disjoint odd circuits in bridgeless cubic graphs}\label{section odd}

In \cite{s4 gm jp, zerafa thesis}, a stronger problem than the $S_{4}$-Conjecture is proposed.

\begin{problem}[\cite{s4 gm jp, zerafa thesis}]\label{problem given a Perfect Matching S4}
Establish whether any perfect matching of a bridgeless cubic graph can be extended to an $S_{4}$-pair.
\end{problem}

It is known that not every perfect matching can be extended to an FR-triple and consequently, neither to a Fulkerson cover. In fact, consider the Petersen graph $P$ and expand each of its vertices into a triangle (see Figure \ref{figure Ptriangles S4 yes FR no}). Let the resulting graph on $30$ vertices be denoted by $P'$, and let $N$ be the set of edges in $P'$ corresponding to $E(P)$. The set $N$ is a perfect matching of $P'$, and it is not difficult to see that it cannot be extended to an FR-triple or to a Fulkerson cover. However, as can be seen in Figure \ref{figure Ptriangles S4 yes FR no}, $N$ (shown as dashed lines) can in fact be extended to an $S_{4}$-pair of $P'$ (with the edges of the second perfect matching shown as dotted lines). 

\begin{figure}[ht]
      \centering
      \includegraphics[scale=1]{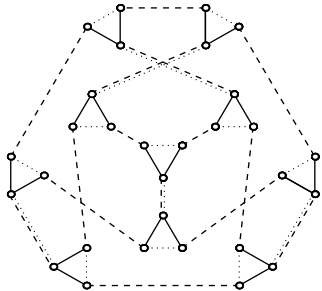}
      \caption{An example of an $S_{4}$-pair in the graph $P'$ obtained from the Petersen graph $P$ by replacing each vertex by a triangle.}
      \label{figure Ptriangles S4 yes FR no}
\end{figure}

Moreover, saying that Problem \ref{problem given a Perfect Matching S4} is true for every bridgeless cubic graph is equivalent to the following statement: for every bridgeless cubic graph $G$, given any collection $\mathcal{O}$ of disjoint odd circuits of $G$, there exists a perfect matching $M$ such that $M\cap E(C)\neq \emptyset$, for every $C\in\mathcal{O}$ (see \cite{s4 gm jp, zerafa thesis}). One implication is clearly obvious. Thus, assume that every perfect matching of any bridgeless cubic graph can be extended to an $S_4$-pair, and consider a collection of disjoint odd circuits in a bridgeless cubic graph $G$. Expand every vertex not covered by the circuits in the collection to a triangle, and let the resulting bridgeless cubic graph be denoted by $G'$. The initial odd circuits and all the new expanded triangles give a $2$-factor $R$ of $G'$. 
   Let $N'=E(G')\setminus R$. By our assumption, there exists a perfect matching $M'$ such that $G'\setminus (N'\cup M')$ is bipartite, implying that $M'$ intersects all the odd circuits in $R$, including all the new expanded triangles. Let $M'_{\triangle}$ be the set of edges belonging simultaneously to $M'$ and the new expanded triangles. 
   One can immediately see that the set of the edges of $G$ corresponding to $M'\setminus M'_{\triangle}$ is a perfect matching of $G$ intersecting all the odd circuits in the initial collection of odd circuits in $G$, proving the equivalence of the two statements.
 
\subsection{Main result}

We now proceed to proving our main result, and consequently solve Problem \ref{problem given a Perfect Matching S4} and the $S_{4}$-Conjecture. 

\begin{theorem}
Let $G$ be a bridgeless cubic graph. Let $F$ be a $1^+$-factor of $G$ and let $e\in E(G)$. Then, there exists a perfect matching $M$ of $G$ such that $e\in M$, and $F$ and $M$ are a quelling pair.
\end{theorem}

\begin{proof}
Let $G$ be a minimum counterexample to the above statement. Let $e\in E(G)$ be an edge of $G$ such that there exists a $1^+$-factor which cannot be extended to a quelling pair by a perfect matching containing $e$. Amongst all such $1^+$-factors of $G$, we can choose an inclusion-wise maximal one, denoted by $F$. By the choice of $F$, we may assume that $e\in F$ and that $G\setminus F$ contains only (disjoint) odd circuits, that is, $F$ is a join. Let the set of components of $G\setminus F$ be denoted by $\mathcal{O}$. 

If $G$ is 3-edge-colourable, then its edge set decomposes into three perfect matchings. It suffices to take $M$ to be the one containing $e$. Then $G\setminus M$ is a union of even circuits, and so $F$ and $M$ are a quelling pair by definition. 
Therefore, $G$ is not 3-edge-colourable, and hence, it is not bipartite. In particular, $G$ is not the unique bipartite cubic graph on two vertices. 

It is easy to see that we may assume that $G$ is connected.
\\

\noindent\textbf{Claim 1.} The graph $G$ does not have any 2-edge-cuts.

\noindent\emph{Proof of Claim 1.} Suppose that $G$ admits a 2-edge-cut $E(V',V^{\prime\prime})$ with 
$E(V',V'')=\{f_{1}, f_{2}\}=:X$, 
where $f_{1}=v'_{1}v''_{1}$ and $f_{2}=v'_{2}v''_{2}$ for some $v'_{1},v'_{2}\in V'$ and $ v''_{1},v''_{2} \in V''$. Since $G$ is bridgeless, $v'_1\neq v'_2$ and $v''_1\neq v''_2$. See Figure \ref{figure 2cut} for an illustration.

\begin{figure}[ht]
      \centering
      \includegraphics[scale=1]{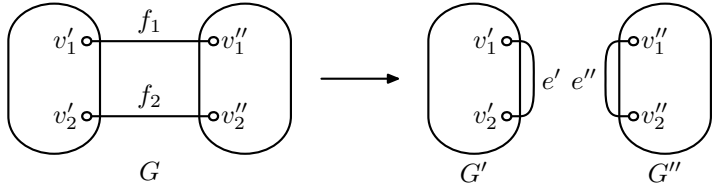}
      \caption{The graphs $G'$ and $G''$ when $G$ admits a 2-edge-cut $\{f_1,f_2\}$.}
      \label{figure 2cut}
\end{figure}

By the choice of $F$, we have that either $X$ is contained in the edge set of some odd circuit $C_X$ in $\mathcal{O}$, or $X\subset F$. Let $G'$ and $G''$ be the two graphs on the vertex sets $V'$ and $V''$ obtained from $G$ after deleting $X$ and adding the edges $e'=v'_{1}v'_{2}$ and $e''=v''_{1}v''_{2}$, respectively. Let 
\[F' =
  \begin{cases}
  F\cap E(G') & $if $F\cap X=\emptyset$,$\\
  \left(F\cap E(G')\right)\cup \{e'\} & $otherwise;$
  \end{cases}\]
and let 
\[F'' =
  \begin{cases}
  F\cap E(G'') & $if $F\cap X=\emptyset$,$\\
  \left(F\cap E(G'')\right)\cup \{e''\} & $otherwise.$
  \end{cases}\]
Clearly, $F'$ ($F''$) is a $1^+$-factor of $G'$ ($G''$, respectively). 

Without loss of generality, we may assume that at least one of the endvertices of $e$ corresponds to a vertex of $G'$. First, consider the case when $e\in X$. Observe that by our choice of $F$, $X\subset F$. By minimality of $G$, there exists a perfect matching $M'$ of $G'$ ($M''$ of $G''$) containing $e'$ ($e''$) such that $F'$ and $M'$ ($F''$ and $M''$) are a quelling pair of $G'$ (of $G''$, respectively). Consequently, $M=M'\cup M''\cup X \setminus \{e',e''\}$ is a perfect matching of $G$ containing $e$. Moreover, since $X\subset M$, $F$ and $M$ are a quelling pair of $G$, contradicting our initial assumption.

Therefore, $e$ corresponds to an edge in $G'\setminus e'$. For simplicity, we shall refer to the edge in $G'$ corresponding to $e$, with the same name, that is, by $e$. By minimality of $G$, there exists a perfect matching $M'$ of $G'$ containing $e$ such that $F'$ and $M'$ are a quelling pair of $G'$. 

If $e'\in M'$, then let $M''$ be a perfect matching of $G''$ containing $e''$ such that $F''$ and $M''$ are a quelling pair of $G''$. Such a perfect matching $M''$ exists by minimality of $G$. Once again, $M=M'\cup M''\cup X \setminus \{e',e''\}$ is a perfect matching of $G$ containing $e$, and since $X\subset M$, $F$ and $M$ are a quelling pair of $G$. Thus, $e'\notin M'$.

In this case, let $e'''$ be an edge adjacent to $e''$ in $G''$. By minimality of $G$, there exists a perfect matching $M''$ of $G''$ containing $e'''$ such that $F''$ and $M''$ are a quelling pair of $G''$. Since $e''\notin M''$, $M=M'\cup M''$ is a perfect matching of $G$ containing $e$. It remains to prove that $F$ and $M$ are a quelling pair of $G$. Every odd circuit $C\ne C_X$ in $\mathcal{O}$ corresponds to an odd circuit either in $G'\setminus F'$ or in $G''\setminus F''$. The odd circuit $C_X$ (if it exists), corresponds to two circuits $C'_X$ and $C''_X$ in $G'\setminus F'$ and $G'' \setminus F''$, respectively, of different parity. Therefore, at least one of them (the odd one) is hit, that is, intersected, by at least one edge of $M'$ or $M''$ in $G'$ or $G''$, respectively, and thus, $C_X$ is hit by at least one edge of $M$ in $G$. Indeed, $F$ and $M$ are a quelling pair of $G$, a contradiction.
\hfill {\tiny$\blacksquare$}\\

In particular, Claim 1 implies that $G$ does not have any parallel edges, that is, $G$ is simple.\\

\noindent\textbf{Claim 2.} The graph $G$ does not have any cyclic 3-edge-cuts.

\noindent\emph{Proof of Claim 2.} Suppose that $G$ admits a cyclic 3-edge-cut $E(V',V^{\prime\prime})$ with 
$E(V',V'')=\{f_{1}, f_{2}, f_{3}\}=:X$, where each $f_{i}=v'_{i}v''_{i}$, for some $v'_{1},v'_{2},v'_{3}\in V'$ and $ v''_{1},v''_{2}, v''_{3} \in V''$. Since $G$ has no 2-edge-cuts, the vertices $v'_{1},v'_{2},v'_{3},  v''_{1},v''_{2}, v''_{3}$ are all  distinct.

By the choice of $F$, the parity of $|F\cap X|$ is odd, that is, either $|F\cap X|=3$, meaning there is no odd circuit in $\mathcal{O}$ intersecting $X$, or $|F\cap X|=1$, meaning the cut $X$ is intersected by a unique circuit $C_X$ in $\mathcal{O}$. Without loss of generality, we shall assume that when $|F\cap X|=1$, $F\cap X=f_{1}$.

Let $G'$ and $G''$ be the two graphs obtained from $G$ after deleting $X$ and joining the vertices $v'_{i}$ to a new vertex $v'$, and the vertices $v''_{i}$ to a new vertex $v''$. For each $i\in\{1,2,3\}$, let $e'_{i}=v'_{i}v'$ and $e''_{i}=v''_{i}v''$. 

\begin{figure}[ht]
      \centering
      \includegraphics[width=0.8\textwidth]{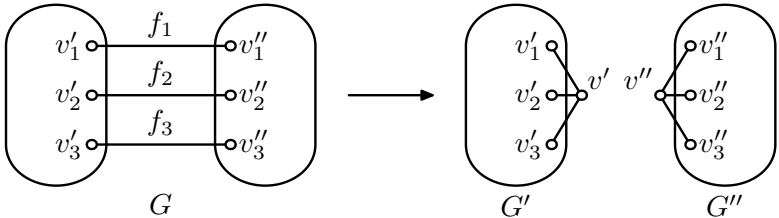}
      \caption{The graphs $G'$ and $G''$ when $G$ admits a cyclic 3-edge-cut $\{f_1,f_2,f_3\}$.}
      \label{figure 3cut}
\end{figure}

Let

\[F' =
  \begin{cases}
  \left(F\cap E(G')\right)\cup \{e'_{1}\} & $if $|F\cap X|=1$,$\\
  \left(F\cap E(G')\right)\cup\{e'_{1},e'_{2},e'_{3}\} & $otherwise.$
  \end{cases}\]

Similarly, let
\[F'' =
  \begin{cases}
  \left(F\cap E(G'')\right)\cup \{e''_{1}\} & $if $|F\cap X|=1$,$\\
  \left(F\cap E(G'')\right)\cup\{e''_{1},e''_{2},e''_{3}\} & $otherwise.$
  \end{cases}\]

It is not hard to see that $F'$ ($F''$) is a $1^+$-factor of $G'$ (of $G''$, respectively). Every odd circuit $C\ne C_X$ in $\mathcal{O}$ corresponds to an odd circuit either in $G'\setminus F'$ or in $G''\setminus F''$. The odd circuit $C_X$ (if it exists) corresponds to two circuits $C'_X$ and $C''_X$ in $G'\setminus F'$ and $G'' \setminus F''$, respectively, having different parity.

Without loss of generality, we may assume that at least one of the endvertices of $e$ corresponds to a vertex in $V'$. 
We consider two cases, depending on the size of $|F\cap X|$. 

\textbf{Case A.} First, consider the case when $|F\cap X|=3$. When $e\in X$, say $e=f_{1}$, then, by minimality of $G$, there exists a perfect matching $M'$ of $G'$ ($M''$ of $G''$) containing $e'_{1}$ ($e''_{1}$), such that $F'$ and $M'$ ($F''$ and $M''$) form a quelling pair of $G'$ (of $G''$, respectively). Consequently, $M=M'\cup M''\cup \{f_{1}\}\setminus\{e'_{1},e''_{1}\}$ is a perfect matching of $G$ containing $e=f_1$, and moreover, $F$ and $M$ are a quelling pair of $G$. 

It remains to consider the case when $e\notin X$, and so the endvertices of $e$ both correspond to vertices in $G'$. Once again, for simplicity, we shall refer to this edge as $e$. Let $M'$ be a perfect matching of $G'$ containing $e$ such that $F'$ and $M'$ are a quelling pair of $G'$. Without loss of generality, assume that $e'_{1}\in M'$. Let $M''$ be a perfect matching of $G''$ containing $e''_{1}$ such that $F''$ and $M''$ are a quelling pair of $G''$. Let $M=M'\cup M''\cup \{f_{1}\}\setminus \{e'_{1},e''_{1}\}$. This is a perfect matching of $G$ containing $e$, and as before, $F$ and $M$ form a quelling pair of $G$, a contradiction. 

\textbf{Case B.} Suppose that $|F\cap X|=1$. 
When $e\in X$, say $e=f_{i}$, then, by minimality of $G$, there exists a perfect matching $M'$ of $G'$ ($M''$ of $G''$) containing $e'_{i}$ ($e''_{i}$), such that $F'$ and $M'$ ($F''$ and $M''$) form a quelling pair of $G'$ (of $G''$, respectively). Consequently, $M=M'\cup M''\cup \{f_{i}\}\setminus\{e'_{i},e''_{i}\}$ is a perfect matching of $G$ containing $e=f_i$.
To observe that $F$ and $M$ form a quelling pair of $G$, it suffices to note that the only odd circuit in $\mathcal{O}$  possibly not intersected by $M$ is $C_X$. However, this can only happen if $i=1$, and, if this is the case, then
either $C'_X$ or $C''_X$ is odd, and so it is hit by an edge of $M'$ or $M''$, not incident to $v'$ or $v''$. However, this means that $C_X$ is also hit by the corresponding edge of $M$ in $G$, a contradiction.

It remains to consider the case when $e\notin X$, and so the endvertices of $e$ both correspond to vertices in $G'$. Once again, for simplicity, we shall refer to this edge as $e$. Let $M'$ be a perfect matching of $G'$ containing $e$ such that $F'$ and $M'$ are a quelling pair of $G'$. We have $e'_{i}\in M'$ for some $i\in\{1,2,3\}$. Let $M''$ be a perfect matching of $G''$ containing $e''_{i}$ such that $F''$ and $M''$ are a quelling pair of $G''$. Let $M=M'\cup M''\cup \{f_{i}\}\setminus \{e'_{i},e''_{i}\}$. This is a perfect matching of $G$ containing $e$. As before, $F$ and $M$ form a quelling pair of $G$, unless $i=1$ and no edge of $G'$ or $G''$ corresponding to an edge of $C_X$ is hit by $M'$ or $M''$, which is impossible since either $C'_X$ or $C''_X$ is an odd circuit. \hfill {\tiny$\blacksquare$}\\

From this point on we may assume that $G$ has no cyclic 3-edge-cuts. Therefore, it is cyclically 4-edge-connected, unless $G=K_4$, but for this particular graph it is easy to see that it is not a counterexample, since for every edge $e$, the complement of the only perfect matching containing $e$ is an even circuit.  


We now consider the edges at distance $2$ from $e$ (distance measured as the distance of the corresponding vertices in the line graph of $G$). 

\noindent\textbf{Claim 3.} Let $f$ be an edge at distance 2 from $e$. Then, $f\in F$. 

\noindent\emph{Proof of Claim 3.} 
We will use a procedure that transforms a cubic graph $G$ into a cubic graph $G'$ smaller than $G$, such that every perfect matching of $G'$ containing a certain edge can be extended into a perfect matching of $G$ containing the corresponding edge. This operation was already used by Voorhoeve \cite{voorhoeve} to study perfect matchings in bipartite cubic graphs and it is a key ingredient for counting perfect matchings in general in \cite{EsperetLovaszPlummer}.

Let $f=uv$, let the neighbours of $u$ distinct from $v$ be $\alpha$ and $\gamma$, and let the neighbours of $v$ distinct from $u$ be $\beta$ and $\delta$. 
In particular, since $G$ is cyclically 4-edge-connected, these four vertices are all distinct. Without loss of generality, we may assume that $\alpha$ is an endvertex of $e$.

\begin{figure}[ht]
      \centering
      \includegraphics[width=0.5\textwidth]{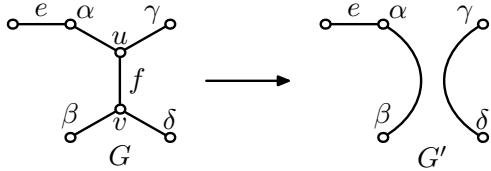}
      \caption{The vertices $\alpha,\beta, \gamma, \delta$ and an $(\alpha\beta:\gamma\delta)_{uv}$-reduction.}
      \label{figure reduction1}
\end{figure}

As shown in Figure \ref{figure reduction1}, we obtain a smaller graph (possibly containing parallel edges) by deleting the endvertices of $f$ (together with all edges incident to them) and adding the edges $\alpha\beta$ and $\gamma\delta$. Let this resulting graph be $G'$. We shall say that $G'$ is obtained after an $(\alpha\beta:\gamma\delta)_{uv}$-reduction. 
It is well-known that when applying this operation, the cyclic edge-connectivity of a cubic graph can drop by at most 2. Since $G$ is cyclically 4-edge-connected, $G'$ is bridgeless. 

Let the edge in $G'$ corresponding to $e$, and the vertices in $G'$ corresponding to $\alpha,\beta, \gamma, \delta$ be denoted by the same name. We recall that any perfect matching of $G'$ which contains $e$ can be extended to a perfect matching of $G$ containing the edge $e$ (see also Figure \ref{figure reduction11}). In fact, let $M'$ be a perfect matching of $G'$ containing $e$. This is extended to a perfect matching $M$ of $G$ containing $e$ as follows:

\[M =
  \begin{cases}
  M'\cup\{u\gamma, v\delta\}\setminus \{\gamma\delta\}  & $if $\gamma\delta\in M'$,$ \\
  M'\cup \{f\} & $otherwise.$ 
  \end{cases}\]

\begin{figure}[ht]
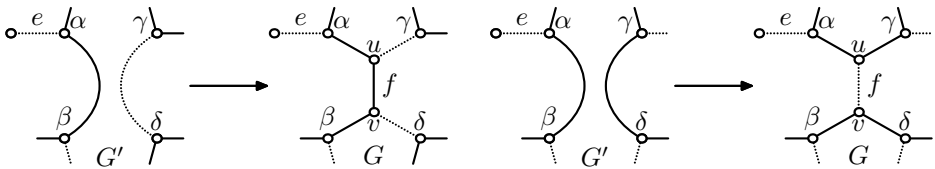

      \centerline{\hfill
      \includegraphics{reduction.2}\hfill\hfill
      \includegraphics{reduction.3}\hfill}
      \caption{Extending a perfect matching of $G'$ containing $e$ to a perfect matching of $G$ containing $e$. Dotted lines represent edges in $M$ or $M'$.}
      \label{figure reduction11}
\end{figure}

Suppose that for some edge $f$ at distance 2 from $e$ we have $f\notin F$. 
By the choice of $F$, $f$ is in some odd circuit $C_f$ in $\mathcal{O}$. This means that exactly one of $u\alpha$ and $u\gamma$, and exactly one of $v\beta$ and $v\delta$ belong to $C_f$. Without loss of generality, we may assume that $u\alpha\in F$ if and only if $v\beta \in F$ (otherwise, we rename $\beta$ and $\delta$). Let $G'$ be the graph obtained from $G$ after an $(\alpha\beta:\gamma\delta)_{uv}$-reduction. Let $F'$ be the $1^+$-factor of $G'$ as follows:

\[F' =
  \begin{cases}
  F\cup \{\alpha\beta\}\setminus \{u\alpha,v\beta\} & 
  $if $\{u\alpha,v\beta\}\subset F$,$\\
  F\cup \{\gamma\delta\}\setminus \{u\gamma,v\delta\} & 
  $otherwise.$
  \end{cases}\]

This is portrayed in Figure \ref{figure casei}.

\begin{figure}[ht]
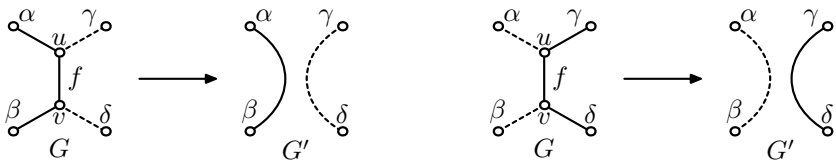

      \centering
      \hfill
      \includegraphics{reduction.4}\hfill\hfill
      \includegraphics{reduction.5}\hfill{}
      \caption{If $f\not\in F$, then we apply induction on $G'$ -- the graph obtained from $G$ after an $(\alpha\beta:\gamma\delta)_{uv}$-reduction. 
      Dashed lines represent edges in $F'$ or $F$. }
      \label{figure casei}
\end{figure}

Observe that every odd circuit $C\neq C_f$ remains the same in $G'$; and that $C_f$ is transformed into an odd circuit $C'_f$ in $G'$, shorter than $C_f$ by two edges.
Since $G'$ is bridgeless and its order is strictly less than $G$, it is not a counterexample. Let $M'$ be a perfect matching of $G'$ containing $e$ such that $F'$ and $M'$ are a quelling pair. We extend this perfect matching to a perfect matching $M$ of $G$ containing $e$ as described above (see Figure \ref{figure reduction11}), and claim that $F$ and $M$ are a quelling pair. Every odd circuit $C'\ne C'_f$ in $G'\setminus F'$ is hit by an edge of $M'$ in $G'$, and so the corresponding circuit $C$ is hit by the corresponding edge of $M$ in $G$. The odd circuit $C'_f$ is hit by an edge $M'$ in $G'$, and so the corresponding circuit $C_f$ is hit by the corresponding edge in $G$, unless the hitting edge is $\alpha\beta$ (or $\gamma\delta$), but then $C_f$ is hit by at least one edge (in fact, two edges) on the path from $\alpha$ to $\beta$ (or from $\gamma$ to $\delta$, respectively) which contains $u$ and $v$. Hence, $F$ and $M$ are a quelling pair of $G$, a contradiction. \hfill {\tiny$\blacksquare$}\\

From this point on we may assume that for every edge $f$ at distance 2 from $e$ we have $f\in F$. As a consequence, by the maximality of $F$, we have that all the edges at distance at most 1 from $e$ are in $F$, otherwise there would be a component in $\mathcal{O}$ which is not a circuit, a contradiction. Once again, let us consider an edge $f=uv$ at distance 2 from $e$. Let $\alpha$, $\beta$, $\gamma$ and $\delta$ be as defined in Claim 3. The edge $u\gamma$ is at distance 2 from $e$, and so $u\gamma\in F$, implying that $u\alpha\in F$ as well. Furthermore, it can be easily seen that $v\beta\in F$ if and only if $v\delta\in F$.

Consider first the case when $\{v\beta,v\delta\}\subset F$. Then, in the graph $G'$, obtained by an $(\alpha\beta:\gamma\delta)_{uv}$-reduction, the $1^+$-factor $F'=F \cup \{\alpha\beta,\gamma\delta\}\setminus\{u\alpha,v\beta,uv,u\gamma,v\delta\}$ can be extended to a quelling pair of $G'$ by a perfect matching $M'$ containing $e$. It is easy to see that the perfect matching $M$ (of $G$) containing $e$ obtained as an extension of $M'$ forms a quelling pair with $F$ in $G$, a contradiction. This implies that $v\beta$ and $v\delta$ do not belong to $F$. 

When $\{v\beta,v\delta\}\cap F = \emptyset$, we cannot use the reduction portrayed in Figure \ref{figure reduction1} as we do not have a guarantee that we can obtain a perfect matching $M$ intersecting the odd circuit in $\mathcal{O}$ containing the edges $v\beta$ and $v\delta$. Since $G$ is cyclically 4-edge-connected, this latter odd circuit is of length at least 5. Let $\delta, v, \beta, y,z $ be consecutive and distinct vertices on this odd circuit (see Figure \ref{figure reduction2}). Moreover, let $w$ and $x$ be the vertices in $G$ such that $\{w\beta, xy\}\subset F$. We proceed by applying an $(\alpha\beta:\gamma\delta)_{uv}$-reduction followed by an $(\alpha x:wz)_{\beta y}$-reduction as portrayed in Figure \ref{figure reduction2}.

\begin{figure}[ht]
      \centering
      \includegraphics{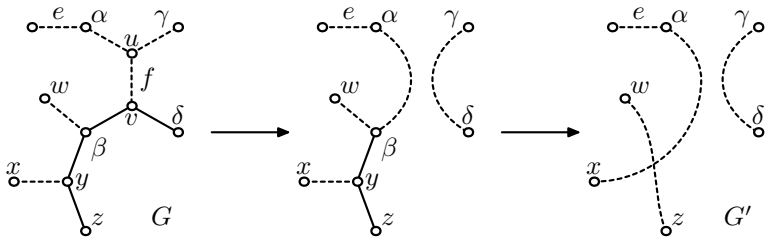}
      \caption{An $(\alpha\beta:\gamma\delta)_{uv}$-reduction followed by an $(\alpha x:wz)_{\beta y}$-reduction.}
      \label{figure reduction2}
\end{figure}

Let the resulting graph after these two reductions be denoted by $G'$. 
We remark that $G'$ may have parallel edges and even loops. Let $F'=F\cap E(G-\{u,v,\beta,y\})\cup \{\alpha x, wz,\gamma\delta\}$. 
Since $G'$ is obtained by applying twice the reduction at an edge at distance 2 from $e$, a perfect matching of $G'$ containing $e$ can always be extended to a perfect matching of $G$ containing $e$ (recall that $\beta y$ cannot be adjacent to $e$, otherwise $\beta y\in F$). 
Moreover, any such matching contains either the edge $\beta y$ or the edge $yz$, and so it hits the odd circuit in $\mathcal{O}$ containing the edges $v\beta$ and $v\delta$. Therefore, as long as $G'$ is bridgeless, by minimality of $G$, there exists a perfect matching $M'$ of $G'$ containing $e$ such that $F'$ and $M'$ are a quelling pair of $G'$. This extends to a perfect matching $M$ of $G$ containing $e$ such that $F$ and $M$ are a quelling pair of $G$. This contradicts our initial assumption that $G$ is a counterexample. Consequently, $G'$ must admit some bridge $g$.
Let $\Omega_{1}=\{\alpha, \gamma,\delta, w,x,z\}$ and let 
$\Omega_{2}=\{u,v,\beta, y\}$.

In order to prove that $G'$ is bridgeless, we prove that it is 2-edge-connected (which is a stronger property). Before considering an edge-cut of size one, we prove that the graph $G'$ is connected.
For, suppose that $G'$ is disconnected. Then $G'$ admits at least two components. Label the vertices of a first component by $A$, and the remaining vertices by $B$. This shall result in a labelling of the vertices of $G'$ such that each edge in $G'$ is monochromatic (an edge in $G'$ is said to be monochromatic if its endvertices have the same label). For the three edges $\alpha x$, $wz$, and $\gamma\delta$, either their endpoints receive all the same label, or two of them receive one label and the remaining four receive the other. We first extend this labelling of $V(G')$ to a partial labelling of $V(G)$ by giving to the vertices in $V(G)-\Omega_{2}$ the same label they had in $G'$.
Next, in both cases, we label all the vertices in $\Omega_2$ with the label appearing four or six times in $\Omega_1$, thus obtaining a labelling of $G$ such that there are at most two edges which are not monochromatic. However, this means that $G$ is either disconnected or contains a 2-edge-cut, a contradiction.

Therefore, $G'$ must be connected. We label the vertices of $G'\setminus \{g\}$ with labels $A$ and $B$ depending in which connected component of $G'\setminus\{g\}$ they belong to.
Consequently, $G'$ has exactly one edge which is not monochromatic: the bridge $g$. We consider different cases depending on the number of vertices in $\Omega_{1}$ labelled $A$ in $G'$, and show that, in each case, a bridge in $G'$ would imply that $G$ is not cyclically 4-edge-connected.  Without loss of generality, we shall assume that the number of vertices in $\Omega_{1}$ labelled $A$ is at least the number of vertices in $\Omega_{1}$ labelled $B$. We consider four cases.

\begin{itemize}

\item[(B0)] All the vertices in $\Omega_{1}$ are labelled $A$ in $G'$.

As above, we extend this labelling to a partial labelling of $G$, and then give label $A$ to all the vertices in $\Omega_{2}$. However, this means that $G$ has exactly one edge, corresponding to $g$, which is not monochromatic, a contradiction, since $G$ does not admit a bridge.
 
\item[(B1)] Exactly 5 vertices in $\Omega_{1}$ are labelled $A$ in $G'$.

This means that one of the edges in $\{\alpha x, wz, \gamma\delta\}$ is the bridge $g$. Once again, we extend this labelling to a partial labelling of $G$, and then give label $A$ to all the vertices in $\Omega_{2}$. However, this means that $G$ has an edge which has exactly one endvertex in $\Omega_{1}$ labelled $B$ and exactly one endvertex in $\Omega_{2}$ labelled $A$, implying that $G$ admits a bridge, a contradiction once again. 

\item[(B2)] Exactly 4 vertices in $\Omega_{1}$ are labelled $A$ in $G'$.

Since $G'$ has exactly one edge which is not monochromatic, the edges in $\{\alpha x, wz,\gamma\delta\}$ are all monochromatic.
As in the previous cases, we extend this labelling to a partial labelling of $V(G)$, and then give label $A$ to all the vertices in $\Omega_{2}$. However, this means that $G$ has exactly two edges each having exactly one endvertex in $\Omega_{1}$ labelled $B$ and exactly one endvertex in $\Omega_{2}$ labelled $A$. These two edges together with the edge in $G$ corresponding to the bridge $g$ form a 3-edge-cut in $G$. Since $G$ has no cyclic 3-edge-cuts, this cut must be trivial -- it separates a singleton from the rest of the graph. Therefore, one of the edges $\alpha x$, $wz$, or $\gamma\delta$ is a loop in $G'$, and so $\alpha=x$, $w=z$, or $\gamma=\delta$ in $G$. If $w=z$ or $\gamma=\delta$, then the vertices $w,\beta, y$, or $\gamma, u,v$ would induce a triangle in $G$, which is impossible since $G$ is cyclically 4-edge-connected. If $\alpha =x$, then the edges $\beta y$ and $yz$ (which are not in $F$) would be at distance at most 2 from $e$, a contradiction once again.

\item[(B3)] Exactly 3 vertices in $\Omega_{1}$ are labelled $A$ in $G'$.

Since $G'$ has exactly one edge which is not monochromatic, there is exactly one edge in $\{\alpha x, wz,\gamma\delta\}$ which is not monochromatic.
As in the preceding cases, we extend this labelling to a partial labelling of $V(G)$, and then give label $A$ to all the vertices in $\Omega_{2}$. However, this means that $G$ has exactly three edges each having exactly one endvertex in $\Omega_{1}$ labelled $B$ and exactly one endvertex in $\Omega_{2}$ labelled $A$, and so these three edges separate a singleton from the rest of the graph. In particular, it means that either $\alpha=x$, $w=z$ or $\gamma=\delta$, which is impossible as we have seen in the previous case.
\end{itemize}

Thus, $G'$ is bridgeless, a contradiction to our assumption.
\end{proof}

Here are some consequences of our main result.

\begin{corollary}
Let $G$ be a bridgeless cubic graph. For every parity subgraph $J$ of $G$ there exists a perfect matching $M$ such that $J$ and $M$ are a quelling pair.
\end{corollary}

\begin{corollary}
Let $G$ be a bridgeless cubic graph and let $\mathcal{O}$ be a collection of disjoint odd circuits of $G$. Then, there exists a perfect matching $M$ such that $M\cap E(C)\neq \emptyset$, for every $C\in\mathcal{O}$.
\end{corollary}

\begin{corollary}
Every bridgeless cubic graph admits an $S_{4}$-pair.
\end{corollary}




\begin{thebibliography}{00}

\bibitem{EsperetLovaszPlummer}
L. Esperet, F. Kardo\v{s}, A.D. King, D. Kr\'{a}l$\!$' and S. Norine,
Exponentially many perfect matchings in cubic graphs,
\emph{Adv. Math.} \textbf{227} (2011), 1646--1664.

\bibitem{FanRaspaud}
G. Fan and A. Raspaud,
Fulkerson's Conjecture and circuit covers,
\emph{J. Combin. Theory Ser. B} \textbf{61(1)} (1994), 133--138.

\bibitem{BergeFulkerson}
D.R. Fulkerson,
Blocking and anti-blocking pairs of polyhedra,
\emph{Math. Program.} \textbf{1(1)} (1971), 168--194.

\bibitem{HolroydSkoviera}
F. Holroyd and M. \v{S}koviera, Colouring of cubic graphs by Steiner triple systems, \emph{J. Comb. Theory, Ser. B} \textbf{91(1)} (2004), 57--66.

\bibitem{Jaeger}
F. Jaeger,
Nowhere-zero flow  problems,
in: L.W. Beineke, R.J. Wilson (eds.),
\emph{Selected  Topics  in  Graph  Theory 3},
San Diego, CA, 1988, 71--95.

\bibitem{vh1}
A. Hakobyan and V.V. Mkrtchyan, 
On Sylvester colorings of cubic graphs,
\emph{Australas. J. Comb.} \textbf{72(3)} (2018), 472--491.

\bibitem{vh2}
A. Hakobyan and V.V. Mkrtchyan, 
$S_{12}$ and $P_{12}$-colorings of cubic graphs,
\emph{Ars Math. Contemp.} \textbf{17(2)} (2019), 431--445.



\bibitem{kaiserraspaud}
T. Kaiser and A. Raspaud,
Perfect matchings with restricted intersection in cubic graphs,
\emph{European J. Combin.} \textbf{31} (2010), 1307--1315.

\bibitem{LovaszPlummer}
L. Lov\'{a}sz and M.D. Plummer,
\emph{Matching Theory}, first ed.,
Elsevier Science, Amsterdam, 1986.

\bibitem{MacajovaSkovieraOddCuts}
E. M\'{a}\v{c}ajov\'{a} and M. \v{S}koviera,
Fano colourings of cubic graphs and the Fulkerson Conjecture,
\emph{Theoret. Comput. Sci.} \textbf{349} (2005), 112--120.

\bibitem{MacajovaSkovieraOddness2}
E. M\'{a}\v{c}ajov\'{a} and M. \v{S}koviera,
Sparsely intersecting perfect matchings in cubic graphs,
\emph{Combinatorica} \textbf{34(1)} (2014), 61--94.

\bibitem{MazzuoccoloS4}
G. Mazzuoccolo,
New conjectures on perfect matchings in cubic graphs,
\emph{Electron. Notes Discrete Math.} \textbf{40} (2013), 235--238.

\bibitem{MazzuoccoloEquivalence}
G. Mazzuoccolo,
The  equivalence  of  two  conjectures  of  Berge  and  Fulkerson,
\emph{J. Graph Theory} \textbf{68} (2011), 125--128.

\bibitem{gmgtjp} 
G. Mazzuoccolo, G. Tabarelli and J.P. Zerafa, \emph{On the existence of graphs which can colour every regular graph}, accepted for publication in \emph{Discrete Appl. Math.} subject to minor revision, 2021+, \url{https://arxiv.org/abs/2110.13684}.


\bibitem{s4 gm jp}
G. Mazzuoccolo and J.P. Zerafa,
An equivalent formulation of the Fan--Raspaud Conjecture
and related problems, \emph{Ars Math. Contemp.} \textbf{18} (2020), 87--103.

\bibitem{VahanPetersen}
V.V. Mkrtchyan,
A remark on the Petersen coloring conjecture of Jaeger,
\emph{Australas. J. Combin.} \textbf{56} (2013), 145--151.

\bibitem{Vahan}
V.V. Mkrtchyan and G.N. Vardanyan,
On two consequences of \mbox{Berge-Fulkerson} conjecture,
\emph{AKCE Int. J. Graphs Comb} (2019), DOI:\href{https://doi.org/10.1016/j.akcej.2019.03.018}{10.1016/j.akcej.2019.03.018}.

\bibitem{petersen}
J. Petersen, 
Die Theorie der regul\"{a}ren graphs,
\emph{Acta Mathematica} \textbf{15} (1891), 193--220.

\bibitem{voorhoeve}
M. Voorhoeve, A lower bound for the permanents of certain $(0,1)$-matrices, \emph{Nederl. Akad. Wetensch. Indag. Math.} \textbf{41} (1979), 83--86.

\bibitem{zerafa thesis}
J.P. Zerafa, \emph{On the consummate affairs of perfect matchings}, PhD Thesis, Universit\`a degli Studi di Modena e Reggio Emilia, Italy, 2021, \url{https://hdl.handle.net/11380/1237629}.

\end{thebibliography}


\end{document}